\documentclass[twoside]{aiml18}

\usepackage{aiml18macro}

\usepackage{graphicx}
\usepackage{amsmath}
\usepackage{amssymb}
\usepackage{amsfonts,euscript,MnSymbol}
\usepackage{mathtools}
\usepackage{bussproofs}

\EnableBpAbbreviations

\newtheorem{propos}{Proposition}

\def\lastname{Savateev and Shamkanov}

\begin{document}

\begin{frontmatter}
  \title{Cut Elimination for Weak Modal Grzegorczyk Logic via Non-Well-Founded Proofs}
  \author{Yury Savateev}\footnote{This work is supported by the Russian Foundation for Basic Research, grant 15-01-09218a.}
  \address{National Research University Higher School of Economics}
  \author{Daniyar Shamkanov}
  \address{Steklov Mathematical Institute of the Russian Academy of Sciences \\ National Research University Higher School of Economics}

  \begin{abstract}
  We present a sequent calculus for the weak Grzegorczyk logic $\mathsf{Go}$ allowing non-well-founded proofs and obtain the cut-elimination theorem for it by constructing a continuous cut-elimination mapping acting on these proofs.  
  \end{abstract}

  \begin{keyword}
  non-well-founded proofs, weak Grzegorczyk logic, logic Go, cut-elimination, cyclic proofs.
  \end{keyword}
 \end{frontmatter}

%\section{Instructions for authors}
%Important information pertaining to the preparation of your paper:
%\begin{itemize}
%  \item Please prepare your paper by editing this file (\verb|aiml18.tex|) as well as the bibliography file, \verb|aiml18.bib|.
%  \item Please do not use includegraphics on any PDF files not having all fonts embedded. If in doubt, please convert to JPG before including.
%  \item It is strictly prohibited to tamper with the text dimensions, including adding to the text height or width. The dimensions of the style files are already taken to the limit of what the printer of the final proceedings is able to handle. Also, please refrain from using any kind of negative white space. 
%  \item Please make sure too split lines that are too wide in the final output (avoid overfull \verb|\hbox|es).
%\end{itemize}

\section{Introduction}
The logic $\mathsf{Go}$, also known as the weak Grzegorzyk logic, is the smallest normal modal logic containing the axiom $K$ and the axioms $\Box A\to \Box\Box A$ and $\Box(\Box(A\to \Box A)\to A)\to\Box A$. A survey of results on $\mathsf{Go}$ can be found in \cite{Litak}. The logic is sound and complete with respect to the class of transitive frames with no proper clusters and infinite ascending chains \cite{Gore}, and it is a proper sublogic of both G\"odel-L\"ob logic $\mathsf{GL}$ (also known as provability logic) and Grzegorzyk logic $\mathsf{Grz}$. 

Recently a new proof-theoretic presentation for the logic $\mathsf{GL}$ in the form of a sequent calculus allowing non-well-founded proofs was given in \cite{Sham,Iemhoff}. Later, the same ideas were applied to the modal Grzegorczyk logic $\mathsf{Grz}$ in \cite{SavSham,SavSham2}, where it allowed to prove several proof-theoretic properties of this logic syntactically. %We present such applications in the present paper.

In this paper we use the same approach for the logic $\mathsf{Go}$. We consider a sequent calculus allowing non-well-founded proofs $\mathsf{Go_\infty}$ and present the cut-elimination theorem for it. We consider the set of non-well-founded proofs of $\mathsf{Go_\infty}$ and various sets of operations acting on theses proofs as ultrametric spaces and define our cut-elimination operator using the Prie\ss-Crampe fixed-point theorem (see \cite{PrCr}), which is a strengthening of the Banach's theorem. 

In \cite{Ramanayake} Gor\'e and Ramanayake remark that their method for cut elimination for the logic $\mathsf{Go}$ is more complex than the similar methods for the logics $\mathsf{GL}$ and $\mathsf{Grz}$. This difference in complexity seems to be present in our approach as well. The proofs of cut-elimination for $\mathsf{Go_\infty}$ and $\mathsf{Grz_\infty}$ turn out to be almost the same, but the system $\mathsf{Go_\infty}$ itself seems to be more complex (it includes rules of arbitrary arity, where $\mathsf{Grz_\infty}$ has at most binary) and the translation from $\mathsf{Go_\infty}$ to the standart system seems to require bigger induction measure. 

\section{Preliminaries}
\label{s2}
In this section we recall the weak Grzegorczyk logic $\mathsf{Go}$ and define an ordinary sequent calculus for it.
  
\textit{Formulas} of $\mathsf{Go}$, denoted by $A$, $B$, $C$, are built up as follows:
$$ A ::= \bot \,\,|\,\, p \,\,|\,\, (A \to A) \,\,|\,\, \Box A \;, $$
where $p$ stands for atomic propositions. 

The Hilbert-style axiomatization of $\mathsf{Go}$ is given by the following axioms and inference rules:

\textit{Axioms:}
\begin{itemize}
\item[(i)] Boolean tautologies;
\item[(ii)] $\Box (A \rightarrow B) \rightarrow (\Box A \rightarrow \Box B)$;
\item[(iii)] $\Box A \rightarrow \Box \Box A$;
\item[(iv)] $\Box(\Box(A \rightarrow \Box A) \rightarrow A) \rightarrow \Box A$.
\end{itemize}

\textit{Rules:} modus ponens, $A / \Box A$. 

Now we define an ordinary sequent calculus for $\mathsf{Go}$. A \textit{sequent} is an expression of the form $\Gamma \Rightarrow \Delta$, where $\Gamma$ and~$\Delta$ are finite multisets of formulas. For a multiset of formulas $\Gamma = A_1,\dotsc, A_n$, we set $\Box \Gamma := \Box A_1,\dotsc, \Box A_n$.

The system $\mathsf{Grz_{Seq}}$, is defined by the following initial sequents and inference rules: 

\begin{gather*}
\AXC{ $\Gamma, A \Rightarrow A, \Delta $ ,}
\DisplayProof \qquad
\AXC{ $\Gamma , \bot \Rightarrow \Delta $ ,}
\DisplayProof
\end{gather*}
\begin{align*}
&
\AXC{$\Gamma , B \Rightarrow \Delta $}
\AXC{$\Gamma \Rightarrow A,\Delta $}
\LeftLabel{$\mathsf{\to_L}$}
\BIC{$\Gamma , A \to B \Rightarrow \Delta$}
\DisplayProof\;,& &
\AXC{$\Gamma , A \Rightarrow B, \Delta $}
\LeftLabel{$\mathsf{\to_R}$}
\UIC{$\Gamma \Rightarrow A \to B, \Delta$}
\DisplayProof\;,\\\\
&
\AXC{$ \Box\Pi, \Pi,\Box(A\to\Box A) \Rightarrow A$}
\LeftLabel{$\mathsf{\Box_{Go}}$}
\UIC{$\Gamma, \Box \Pi \Rightarrow \Box A ,\Delta $}
\DisplayProof \;.
\end{align*}
\begin{center}
\textbf{Fig. 1.} The system $\mathsf{Go_{Seq}}$
\end{center}
The cut rule has the form
\begin{gather*}
\AXC{$\Gamma\Rightarrow A,\Delta$}
\AXC{$\Gamma,A\Rightarrow\Delta$}
\LeftLabel{$\mathsf{cut}$}
\RightLabel{ ,}
\BIC{$\Gamma\Rightarrow\Delta$}
\DisplayProof
\end{gather*}
where $A$ is called the \emph{cut formula} of the given inference.

\begin{lemma} \label{prop}
$\mathsf{Go_{Seq}} + \mathsf{cut}\vdash \Gamma\Rightarrow\Delta$ if and only if $\mathsf{Go} \vdash \bigwedge\Gamma\to\bigvee\Delta $. 
\end{lemma}
\begin{proof}
Standard transformations of proofs.
\end{proof}

\begin{theorem}\label{cutelimgrz}
If $\mathsf{Go_{Seq}} + \mathsf{cut}\vdash \Gamma\Rightarrow\Delta$, then $\mathsf{Go_{Seq}} \vdash \Gamma\Rightarrow\Delta$. 
\end{theorem}

A syntactic cut-elimination for $\mathsf{Go}$ was obtained by R. Gor\'e and R. Ramanayake in \cite{Ramanayake}. 
In this paper, we will give another proof of this cut-elimination theorem. 

\section{Non-well-founded proofs}
\label{s3}
Now we define a sequent calculus for $\mathsf{Go}$ allowing non-well-founded proofs.  

Inference rules and initial sequents of the sequent calculus $\mathsf{Go_\infty}$ have the following form:
%\begin{gather*}
\[
\AXC{ $\Gamma, p \Rightarrow p, \Delta $ ,}
\DisplayProof\qquad
\AXC{ $\Gamma , \bot \Rightarrow  \Delta$ ,}
\DisplayProof 
\]
\[
\AXC{$\Gamma , B \Rightarrow \Delta $}
\AXC{$\Gamma \Rightarrow A,\Delta $}
\LeftLabel{$\mathsf{\to_L}$}
\BIC{$\Gamma , A \to B \Rightarrow \Delta$}
\DisplayProof\;,\qquad
\AXC{$\Gamma , A \Rightarrow B, \Delta $}
\LeftLabel{$\mathsf{\to_R}$}
\UIC{$\Gamma \Rightarrow A \to B, \Delta$}
\DisplayProof\;,
\]
\[
\AXC{$\Box \Pi,\Pi \Rightarrow A_1,\ldots,A_n,\Box A_1,\ldots\Box A_n$}
\AXC{$\Box \Pi, \Pi\Rightarrow A_1$}
\AXC{$\ldots$}
\AXC{$\Box \Pi, \Pi\Rightarrow A_n$}
\QuaternaryInfC{$\Gamma, \Box \Pi\Rightarrow \Box A_1,\ldots,\Box A_n, \Delta$}
\DisplayProof \;
\]
%\end{gather*}
\begin{center}
\textbf{Fig. 2.} The system $\mathsf{Go}_\infty$
\end{center}

The system $\mathsf{Go}_{\infty}+\mathsf{cut}$ is defined by adding the rule ($\mathsf{cut}$) to the system $\mathsf{Go_\infty}$.

The will refer to all but the leftmost premises of the rule ($\Box$) as "right".

An \emph{$\infty$--proof} in $\mathsf{Go}_\infty$ ($\mathsf{Go}_{\infty}+\mathsf{cut}$) is a (possibly infinite) tree whose nodes are marked by
sequents and whose leaves are marked by initial sequents and that is constructed according to the rules of the sequent calculus. In addition, every infinite branch in an $\infty$--proof must pass through a right premise of the rule ($\Box$) infinitely many times. A sequent $\Gamma \Rightarrow \Delta$ is \emph{provable} in $\mathsf{Go}_\infty$ ($\mathsf{Go}_{\infty}+\mathsf{cut}$) if there is an $\infty$--proof in $\mathsf{Go}_\infty$ ($\mathsf{Go}_{\infty}+\mathsf{cut}$) with the root marked by $\Gamma \Rightarrow \Delta$.

%The \emph{main fragment} of an $\infty$--proof is a finite tree obtained from the $\infty$--proof by cutting every branch at the nearest to the root right premise of the rule ($\Box$). The \emph{local height $\lvert \pi \rvert$ of an $\infty$--proof $\pi$} is the length of the longest branch in its main fragment. An $\infty$--proof only consisting of an initial sequent has height 0.

For a multiset of formulas $\Gamma = A_1,\dotsc, A_n$, we set $$\boxtimes \Gamma := A_1,\ldots,A_n,\Box A_1,\dotsc, \Box A_n.$$

Then the rule ($\Box$) can be written as
\begin{gather*}
\AXC{$\boxtimes \Pi \Rightarrow \boxtimes(A_1,\ldots,A_n)$}
\AXC{$\boxtimes \Pi\Rightarrow A_1$}
\AXC{$\ldots$}
\AXC{$\boxtimes \Pi\Rightarrow A_n$}
\QuaternaryInfC{$\Gamma, \Box \Pi\Rightarrow \Box A_1,\ldots,\Box A_n, \Delta$}
\DisplayProof \;.
\end{gather*}

Let us construct an $\infty$--proof of the sequent $\Box(\Box(p \rightarrow \Box p) \rightarrow p) \Rightarrow \Box p$.

Let $F=\Box(p \rightarrow \Box p) \rightarrow p $ and let $\psi$ be the following proof:
\begin{gather*}
\AXC{$\mathsf{Ax}$}
\noLine
\UIC{$\boxtimes F,p\Rightarrow p,\Box p,\Box p,\Box(p\to\Box p)$}
\LeftLabel{$\mathsf{\to_R}$}
\UIC{$\boxtimes F\Rightarrow \boxtimes(p,p\to\Box p)$}
\DisplayProof\:.
\end{gather*}

Let $\phi$ be the following proof part:

\begin{gather*}
\AXC{$\mathsf{Ax}$}
\noLine
\UIC{$\Box F,p\Rightarrow p,\Box p$}
\AXC{$\psi$}
\AXC{$(1)$}
\noLine
\UIC{$F,p,\Box F\Rightarrow \Box p$}
\LeftLabel{$\mathsf{\to_R}$}
\UIC{$\boxtimes F\Rightarrow p\to\Box p$}
\AXC{$(2)$}
\noLine
\UIC{$\boxtimes F\Rightarrow p$}
\LeftLabel{$\Box$}
\TIC{$\Box F\Rightarrow\Box(p\to\Box p),\Box p,p$}
\LeftLabel{$\mathsf{\to_L}$}
\BIC{$\Box F,\Box(p \rightarrow \Box p) \rightarrow p\Rightarrow p,\Box p$}
\DisplayProof
\end{gather*}

Let $\xi$ be the following proof part:
\[
\AXC{$\psi$}
\AXC{$(1)$}
\noLine
\UIC{$F,p,\Box F\Rightarrow \Box p$}
\LeftLabel{$\mathsf{\to_R}$}
\UIC{$\boxtimes F\Rightarrow p\to\Box p$}
\AXC{$(2)$}
\noLine
\UIC{$\boxtimes F\Rightarrow p$}
\LeftLabel{$\Box$}
\TIC{$p,\boxtimes F\Rightarrow \Box p,\Box(p \rightarrow \Box p)$}
\LeftLabel{$\mathsf{\to_R}$}
\UIC{$\boxtimes F \Rightarrow \boxtimes(p \rightarrow \Box p)$}
\DisplayProof
\]
Let $\theta$ be the following proof part:
\begin{gather*}
\AXC{$\mathsf{Ax}$}
\noLine
\UIC{$\Box F,p\Rightarrow p$}
\AXC{$(1)$}
\noLine
\UIC{$\boxtimes F \Rightarrow \boxtimes(p \rightarrow \Box p)$}
\AXC{$(2)$}
\noLine
\UIC{$F,p,\Box F\Rightarrow \Box p$}
\LeftLabel{$\mathsf{\to_R}$}
\UIC{$\boxtimes F\Rightarrow p\to\Box p$}
\LeftLabel{$\Box$}
\BIC{$\Box F \Rightarrow \Box(p \rightarrow \Box p),p$}
\LeftLabel{$\mathsf{\to_L}$}
\BIC{$\Box F, \Box(p \rightarrow \Box p) \rightarrow p \Rightarrow p$}
\DisplayProof
\end{gather*}

An $\infty$--proof of the sequent $\Box(\Box(p \rightarrow \Box p) \rightarrow p) \Rightarrow \Box p$ can be constructed as follows: 
\begin{gather*}
\AXC{$\vdots$}
\noLine
\UIC{$\phi$}
\AXC{$\vdots$}
\noLine
\UIC{$\theta$}
\LeftLabel{$\Box$}
\BIC{$F,p,\Box F\Rightarrow \Box p$}
\AXC{$\vdots$}
\noLine
\UIC{$\theta$}
\BIC{$\phi$}
\AXC{$\vdots$}
\noLine
\UIC{$\phi$}
\AXC{$\vdots$}
\noLine
\UIC{$\theta$}
\LeftLabel{$\Box$}
\BIC{$F,p,\Box F\Rightarrow \Box p$}
\AXC{$\vdots$}
\noLine
\UIC{$\theta$}
\BIC{$\xi$}
\AXC{$\vdots$}
\noLine
\UIC{$\phi$}
\AXC{$\vdots$}
\noLine
\UIC{$\theta$}
\LeftLabel{$\Box$}
\BIC{$F,p,\Box F\Rightarrow \Box p$}
\BIC{$\theta$}
\LeftLabel{$\Box$}
\RightLabel{.}
\BIC{$\Box(\Box(p \rightarrow \Box p) \rightarrow p) \Rightarrow \Box p$}
\DisplayProof\;.
\end{gather*}

The \emph{$n$-fragment} of an $\infty$--proof is a finite tree obtained from the $\infty$--proof by cutting every branch at the $n$th from the root right premise of a $\Box$-rule. The $1$-fragment of an $\infty$--proof is also called its \emph{main fragment}. The \emph{local height $\lvert \pi \rvert$ of an $\infty$--proof $\pi$} is the length of the longest branch in its main fragment. An $\infty$--proof only consisting of an initial sequent has height 0.

The local height of the $\infty$--proof constructed for the sequent $\Box(\Box(p \rightarrow \Box p) \rightarrow p) \Rightarrow \Box p$ equals to 4 and its main fragment has the form
\begin{gather*}
\AXC{$\Box F,p\Rightarrow p,\Box p$}
\AXC{$\Box F,p\Rightarrow p,\Box p,\Box p,\Box(p\to\Box p)$}
\LeftLabel{$\mathsf{\to_R}$}
\UIC{$\Box F\Rightarrow p,p\to\Box p,\Box p,\Box(p\to\Box p)$}
\AXC{}
\LeftLabel{$\Box$}
\BIC{$\Box F\Rightarrow\Box(p\to\Box p),\Box p,p$}
\LeftLabel{$\mathsf{\to_L}$}
\BIC{$\Box F,\Box(p \rightarrow \Box p) \rightarrow p\Rightarrow p,\Box p$}
\AXC{}
\LeftLabel{$\Box$}
\RightLabel{.}
\BIC{$\Box(\Box(p \rightarrow \Box p) \rightarrow p) \Rightarrow \Box p$}
\DisplayProof
\end{gather*}

We denote the set of all $\infty$-proofs in the system $\mathsf{Go}_{\infty} +\mathsf{cut} $ by $\mathcal P$. 
%For any $n \in \mathbb{N}$, the set of all $\infty$-proofs that do not contain applications of the cut rule in their $n$-fragments is denoted by $\mathcal{P}_n$. We also set $\mathcal{P}_0= \mathcal{P}$.

For $\pi, \tau\in\mathcal P$, we write $\pi \sim_n \tau$ if $n$-fragments of these $\infty$-proofs coincide. For any $\pi, \tau\in\mathcal P$, we also set $\pi \sim_0 \tau$.

For every $\infty$-proof $\pi\in\mathcal P$ one of the following holds: 
\begin{enumerate}
\item $\pi$ consists of a single initial sequent.
\item $\pi$ has the form 
\[
\AXC{$\pi_0$}
\noLine
\UIC{$\Gamma , A \Rightarrow B, \Delta $}
\LeftLabel{$\mathsf{\to_R}$}
\UIC{$\Gamma \Rightarrow A \to B, \Delta$}
\DisplayProof
\]
with $\lvert \pi_0 \rvert<\lvert \pi \rvert$. We denote this by $\pi=\mathsf{\to_{R(A\to B)}}(\pi_0)$.
\item $\pi$ has the form 
\[
\AXC{$\pi_0$}
\noLine
\UIC{$\Gamma , B \Rightarrow \Delta $}
\AXC{$\pi_1$}
\noLine
\UIC{$\Gamma \Rightarrow A,\Delta $}
\LeftLabel{$\mathsf{\to_L}$}
\BIC{$\Gamma , A \to B \Rightarrow \Delta$}
\DisplayProof
\]
with $\lvert \pi_0 \rvert,\lvert \pi_1 \rvert<\lvert \pi \rvert$. We denote this by $\pi=\mathsf{\to_{L(A\to B)}}(\pi_0,\pi_1)$.
\item $\pi$ has the form 
\[
\AXC{$\Gamma\Rightarrow A,\Delta$}
\AXC{$\Gamma,A\Rightarrow\Delta$}
\LeftLabel{$\mathsf{cut}$}
\RightLabel{ ,}
\BIC{$\Gamma\Rightarrow\Delta$}
\DisplayProof
\]
with $\lvert \pi_0 \rvert,\lvert \pi_1 \rvert<\lvert \pi \rvert$. We denote this by $\pi=\mathsf{cut_A}(\pi_0,\pi_1)$.
\item $\pi$ has the form 
\[
\AXC{$\boxtimes \Pi \Rightarrow \boxtimes(A_1,\ldots,A_n)$}
\AXC{$\boxtimes \Pi\Rightarrow A_1$}
\AXC{$\ldots$}
\AXC{$\boxtimes \Pi\Rightarrow A_n$}
\QuaternaryInfC{$\Gamma, \Box \Pi\Rightarrow \Box A_1,\ldots,\Box A_n, \Delta$}
\DisplayProof \;.
\]
with $\lvert \pi_0 \rvert<\lvert \pi \rvert$. We denote this by $\pi=\mathsf{\Box_{\Gamma;\Delta}}(\pi_0,\pi_1,\ldots,\pi_n)$.
\end{enumerate}

Now we define two translations that connect ordinary and non-well-founded sequent calculi for the logic $\mathsf{Go}$. 

\begin{lem}\label{AtoA}
We have  $\mathsf{Go}_{\infty} \vdash \Gamma,A\Rightarrow A,\Delta$ for any sequent $\Gamma \Rightarrow \Delta$ and any formula $A$.
\end{lem}
\begin{proof}
Standard induction on the structure of $A$.
\end{proof}

\begin{lem}\label{Go-schema}
We have $\mathsf{Go}_{\infty}\vdash\Box(\Box(A \rightarrow \Box A) \rightarrow A) \Rightarrow \Box A$ for any formula $A$.
\end{lem}
\begin{proof}
Consider an example of $\infty$--proof for the sequent $\Box(\Box(p \rightarrow \Box p) \rightarrow p) \Rightarrow \Box p$ given above. We transform this example into an $\infty$--proof for $\Box(\Box(A \rightarrow \Box A) \rightarrow A) \Rightarrow A$ by replacing $p$ with $A$ and adding required $\infty$--proofs instead of initial sequents using Lemma \ref{AtoA}.  
\end{proof}

Recall that an inference rule is called admissible (in a given proof system) if, for any instance of the rule, the conclusion is provable whenever all premises are provable. 
\begin{lem}\label{weakening}
The rule
\[
\AXC{$\Gamma\Rightarrow\Delta$}
\LeftLabel{$\mathsf{wk}$}
\UIC{$\Pi,\Gamma\Rightarrow\Delta,\Sigma$}
\DisplayProof
\]
is admissible in the systems $\mathsf{Go_{Seq}} $ and $\mathsf{Go}_{\infty} +\mathsf{cut}$.

The rule
\[
\AXC{$\Gamma,\Pi,\Pi\Rightarrow\Delta$}
\LeftLabel{$\mathsf{ctr}$}
\UIC{$\Gamma,\Pi\Rightarrow\Delta$}
\DisplayProof
\]
is admissible in the system $\mathsf{Go_{Seq}}$.
\end{lem}
\begin{proof}
Standard induction on the structure (local height) 
of a proof of $\Gamma\Rightarrow\Delta$.
%Induction on the structure (on the local height) of a proof of $\Gamma\Rightarrow\Delta$ in a standard way.
\end{proof}

\begin{thm}\label{seqtoinfcut}
If $\mathsf{Go_{Seq}}+\mathsf{cut}\vdash\Gamma\Rightarrow\Delta$, then $\mathsf{Go}_{\infty}+\mathsf{cut}\vdash\Gamma\Rightarrow\Delta$.
\end{thm}
\begin{proof}
Assume $\pi$ is a proof of $\Gamma\Rightarrow\Delta$ in $\mathsf{Go_{Seq}}+\mathsf{cut}$. By induction on the size of $\pi$ we prove $\mathsf{Go}_{\infty}+\mathsf{cut}\vdash\Gamma\Rightarrow\Delta$. 

If $\Gamma \Rightarrow \Delta $ is an initial sequent of $\mathsf{Go_{Seq}}+\mathsf{cut}$, then it is provable in $\mathsf{Go}_{\infty}+\mathsf{cut}$ by Lemma \ref{AtoA}.
Otherwise, consider the last application of an inference rule in $\pi$. 

The only non-trivial case is when the proof $\pi$ has the form 
\[
\AXC{$\pi^\prime$}
\noLine
\UIC{$\Box \Pi,\Pi,\Box(A\to\Box A)\Rightarrow A$}
\LeftLabel{$\mathsf{\Box_{Go}}$}
\RightLabel{ ,}
\UIC{$\Sigma,\Box\Pi\Rightarrow \Box A, \Lambda$}
\DisplayProof
\]
where $\Sigma,\Box\Pi = \Gamma$ and $\Box A, \Lambda = \Delta$. By the induction hypothesis there is an $\infty$--proof $\xi$ of $\boxtimes \Pi,\Box(A\to\Box A)\Rightarrow A$ in $\mathsf{Go}_{\infty}+\mathsf{cut}$.

The required $\infty$--proof for $\Sigma,\Box\Pi\Rightarrow \Box A, \Delta$ has the form:
\[
\AXC{$\xi$}
\noLine
\UIC{$\boxtimes \Pi, \Box(A\to\Box A)\Rightarrow A$}
\LeftLabel{$\mathsf{\to_R}$}
\UIC{$\boxtimes \Pi\Rightarrow F$}
\LeftLabel{$\mathsf{wk}$}
\UIC{$\boxtimes \Pi\Rightarrow\boxtimes F$}
\AXC{$\xi$}
%\noLine
%\UIC{$\boxtimes \Pi, \Box(A\to\Box A)\Rightarrow A$}
\LeftLabel{$\mathsf{\to_R}$}
\UIC{$\boxtimes \Pi\Rightarrow F$}
\LeftLabel{$\Box$}
\BIC{$\Sigma,\Box \Pi\Rightarrow\Box F,\Box A,\Lambda$}
\AXC{$\chi$}
\noLine
\UIC{$\Box F\Rightarrow \Box A$}
\LeftLabel{$\mathsf{wk}$}
\UIC{$\Sigma,\Box \Pi,\Box F\Rightarrow \Box A,\Lambda$}
\LeftLabel{$\mathsf{cut}$}
\BIC{$\Sigma,\Box\Pi\Rightarrow \Box A, \Lambda$}
\DisplayProof
\]

where $F= \Box(A \rightarrow \Box A) \rightarrow A$ and $\chi$ is an $\infty$--proof of $\Box F \Rightarrow \Box A$, which exists by Lemma \ref{Go-schema}.

The cases of other inference rules being last in $\pi$ are straightforward, so we omit them.

\end{proof}

%\end{proof}

For a sequent $\Gamma\Rightarrow\Delta$, let $Sub(\Gamma\Rightarrow\Delta)$ be the set of all subformulas of the formulas from $\Gamma \cup\Delta$.
For a finite set of formulas $\Lambda$, let $\Lambda^\ast$ be the set $\{A\to\Box A\mid A\in\Lambda\}$.

\begin{lem} \label{translation}%For any sequent $\Gamma\Rightarrow\Delta$ and a finite set of formulas $\Lambda$, if $\mathsf{Go_\infty}\vdash \Gamma\Rightarrow\Delta$, then $\mathsf{Go_{Seq}} \vdash \Lambda^\ast,\Gamma\Rightarrow\Delta$.
If $\mathsf{Go_\infty}\vdash \Gamma\Rightarrow\Delta$, then $\mathsf{Go_{Seq}} \vdash \Box(\Lambda_1^\ast),\Lambda_2^\ast,\boxtimes\Omega,\Gamma\Rightarrow\Delta$ for any finite sets of formulas $\Lambda_1, \Lambda_2$, and $\Omega$ such that $\Lambda_2\subset\Lambda_1$.
\end{lem}
\begin{proof}
Assume $\pi$ is an $\infty$--proof of the sequent $\Gamma\Rightarrow\Delta$ in $\mathsf{Go}_\infty$ and $\Lambda_1$, $\Lambda_2$, and $\Omega$ are finite sets of formulas, such that $\Lambda_2\subset\Lambda_1$.

We prove that $\mathsf{Go_{Seq}} \vdash \Box(\Lambda_1^\ast),\Lambda_2^\ast,\boxtimes\Omega,\Gamma\Rightarrow\Delta$ by quadruple induction: by induction on the number of elements in the finite set $Sub(\Gamma\Rightarrow\Delta)\backslash \Lambda_1$ with a subinduction on on the number of elements in the finite set $Sub(\Gamma\Rightarrow\Delta)\backslash \Lambda_2$, subinduction on the number of elements in the finite set $Sub(\Gamma\Rightarrow\Delta)\backslash \Omega$, and with subinduction on $\lvert \pi \rvert$. 

%Since the system $\mathsf{Go_\infty}$ has the subformula property, the number of subformulas does not increase when we move from the conclusion of a rule to its premises.

If $\lvert \pi \rvert=0$, then $\Gamma\Rightarrow\Delta$ is an initial sequent. We see that the sequent $\Box(\Lambda_1^\ast),\Lambda_2^\ast,\boxtimes\Omega,\Gamma\Rightarrow\Delta$ is an initial sequent and it is provable in $\mathsf{Go_{Seq}}$.
Otherwise, consider the last application of an inference rule in $\pi$. 

Case 1. Suppose that $\pi=\mathsf{\to_{R(A\to B)}}(\pi_0)$. %has the form
%\[
%\AXC{$\pi_1$}
%\noLine
%\UIC{$\Gamma,A\Rightarrow B,\Sigma$}
%\LeftLabel{$\mathsf{\to_R}$}
%\RightLabel{ ,}
%\UIC{$\Gamma\Rightarrow A\to B,\Sigma$}
%\DisplayProof
%\]
%where $A\to B,\Sigma = \Delta$.
Notice that $\lvert \pi_0 \rvert < \lvert \pi \rvert $. By the induction hypothesis for $\Lambda_2$, $\Lambda_1$, $\Omega$, and $\pi_0$, the sequent $\Box(\Lambda_1^\ast),\Lambda_2^\ast,\boxtimes\Omega,\Gamma,A\Rightarrow B,\Sigma$, where $A\to B,\Sigma = \Delta$, is provable in $\mathsf{Go_{Seq}}$. 
Applying the rule ($\mathsf{\to_R}$) to it, we obtain that the sequent $\Box(\Lambda_1^\ast),\Lambda_2^\ast,\boxtimes\Omega,\Gamma\Rightarrow\Delta$ is provable in $\mathsf{Go_{Seq}}$.

Case 2. Suppose that $\pi=\mathsf{\to_{L(A\to B)}}(\pi_0,\pi_1)$ %has the form
%\[
%\AXC{$\pi_0$}
%\noLine
%\UIC{$\Sigma, B\Rightarrow \Delta$}
%\AXC{$\pi_1$}
%\noLine
%\UIC{$\Sigma \Rightarrow A,\Delta$}
%\LeftLabel{$\mathsf{\to_L}$}
%\RightLabel{ ,}
%\BIC{$\Sigma, A\to B\Rightarrow \Delta$}
%\DisplayProof
%\]
%where $\Sigma, A\to B = \Gamma$. 
We see that $\lvert \pi_0 \rvert < \lvert \pi \rvert $. By the induction hypothesis for $\Lambda_2$, $\Lambda_1$, $\Omega$, and $\pi_0$, the sequent $\Box(\Lambda_1^\ast),\Lambda_2^\ast,\boxtimes\Omega,\Sigma, B\Rightarrow \Delta$, where $\Sigma, A\to B = \Gamma$, is provable in  $\mathsf{Go_{Seq}}$. Analogously, we have $\mathsf{Grz_{Seq}} \vdash \Box(\Lambda_1^\ast),\Lambda_2^\ast,\boxtimes\Omega,\Sigma \Rightarrow A,\Delta$. Applying the rule ($\mathsf{\to_L}$), we obtain that the sequent $\Box(\Lambda_1^\ast),\Lambda_2^\ast,\boxtimes\Omega,\Gamma \Rightarrow\Delta$ is provable in $\mathsf{Go_{Seq}}$.

Case 3. Suppose that $\pi$ has the form
\[
\AXC{$\pi_0$}
\noLine
\UIC{$\boxtimes \Pi \Rightarrow \boxtimes(A_1,\ldots,A_n)$}
\AXC{$\pi_1$}
\noLine
\UIC{$\boxtimes \Pi\Rightarrow A_1$}
\AXC{$\ldots$}
\AXC{$\pi_n$}
\noLine
\UIC{$\boxtimes \Pi\Rightarrow A_n$}
\QuaternaryInfC{$\Phi, \Box \Pi\Rightarrow \Box A_1,\ldots,\Box A_n, \Sigma$}
\DisplayProof
\]
where $\Phi, \Box \Pi = \Gamma$ and $\Box A_1,\ldots,\Box A_n, \Sigma =\Delta$.

Subcase 3.1: For some $i$, we have $A_i\notin\Lambda_1$. We have that the number of elements in $Sub(\Box\Pi,\Pi\Rightarrow A)\backslash(\Lambda_1\cup \{A_i\})$ is strictly less than the number of elements in $Sub(\Phi, \Box \Pi\Rightarrow \Box A_1,\ldots,\Box A_n, \Sigma)\backslash\Lambda_1$. By the induction hypothesis for $\Lambda_1\cup\{A_i\}$, $\Lambda_1$, $\varnothing$ and $\pi_i$, the sequent $\Box(\Lambda_1^\ast),\Box(A_i\to\Box A_i),\Lambda_1^\ast, \Box \Pi,\Pi \Rightarrow A$ is provable in $\mathsf{Go_{Seq}}$. Then we have
\[
\AXC{$\Box(\Lambda_1^\ast),\Box(A_i\to\Box A_i),\Lambda_1^\ast,\Box \Pi,\Pi \Rightarrow A_i$}
\LeftLabel{$\mathsf{\Box_{Go}}$}
\RightLabel{ .}
\UIC{$\Lambda_2^\ast,\Box(\Lambda_1^\ast),\boxtimes\Omega,\Phi, \Box \Pi\Rightarrow \Box A_1,\ldots,\Box A_i,\ldots,\Box A_n, \Sigma$}
\DisplayProof
\]

Subcase 3.2: For all $i$, we have $A_i\in\Lambda_1$, but there is $i$, such that $A_i\notin\Lambda_2$. We have that the number of elements in $Sub(\Box\Pi,\Pi\Rightarrow A_i)\backslash\Lambda_1$ is strictly less than the number of elements in $Sub(\Phi, \Box \Pi\Rightarrow \Box A_1,\ldots,\Box A_n, \Sigma)\backslash\Lambda_2$. By the induction hypothesis for $\Lambda_1$, $\Lambda_1$, $\varnothing$ and $\pi_i$, the sequent $\Box(\Lambda_1^\ast),\Lambda_1^\ast, \Box \Pi,\Pi \Rightarrow A_i$ is provable in $\mathsf{Go_{Seq}}$. Then we have
\[
\AXC{$\Box(\Lambda_1^\ast),\Lambda_1^\ast, \Box \Pi,\Pi \Rightarrow A_i$}
\LeftLabel{$\mathsf{weak}$}
\UIC{$\Box(\Lambda_1^\ast),\Lambda_1^\ast,\Box \Pi,\Pi ,\Box(A_i\to\Box A_i)\Rightarrow A_i$}
\LeftLabel{$\mathsf{\Box_{Go}}$}
\RightLabel{ .}
\UIC{$\Lambda_2^\ast,\Box(\Lambda_1^\ast),\boxtimes\Omega,\Phi, \Box \Pi\Rightarrow \Box A_1,\ldots,\Box A_i,\ldots,\Box A_n, \Sigma$}
\DisplayProof
\]

Subcase 3.3: For all $i$, we have $A_i\in\Lambda_2\subset\Lambda_1$, but there is a formula $F$ in $\Pi$, such that $F\notin\Omega$. We have that the number of elements in $Sub(\Box\Pi,\Pi\Rightarrow A_1)\backslash(\Omega\cup\Pi)$ is strictly less than the number of elements in $Sub(\Phi, \Box \Pi\Rightarrow \Box A_1,\ldots,\Box A_n, \Sigma)\backslash\Omega$. By the induction hypothesis for $\Lambda_1$, $\Lambda_1$, $\Omega\cup\Pi$ and $\pi_1$, the sequent $\Box(\Lambda_1^\ast),\Lambda_1^\ast,\boxtimes\Omega,\boxtimes(\Pi\backslash\Omega), \boxtimes\Pi \Rightarrow A_1$ is provable in $\mathsf{Go_{Seq}}$. Then we have
\[
\AXC{$\Box(\Lambda_1^\ast),\Lambda_1^\ast,\boxtimes\Omega,\boxtimes(\Pi\backslash\Omega), \boxtimes \Pi \Rightarrow A_1$}
\LeftLabel{$\mathsf{weak}$}
\UIC{$\Box(\Lambda_1^\ast),\Lambda_1^\ast,\boxtimes\Omega,\boxtimes(\Pi\backslash\Omega), \boxtimes \Pi,\Box(A\to\Box A) \Rightarrow A_1$}
\LeftLabel{$\mathsf{\Box_{Go}}$}
\RightLabel{ .}
\UIC{$\Lambda_2^\ast,\Box(\Lambda_1^\ast),\Box\Omega,\Omega,\Phi,\Box(\Pi\backslash\Omega),\Box(\Pi\backslash\Omega), \Box (\Pi\cap\Omega)\Rightarrow \Box A_1,\ldots,\Box A_n, \Sigma$}
\LeftLabel{$\mathsf{ctr}$}
\UIC{$\Lambda_2^\ast,\Box(\Lambda_1^\ast),\boxtimes\Omega,\Phi,\Box \Pi\Rightarrow \Box A_1,\ldots,\Box A_n, \Sigma$}
\DisplayProof
\]

Subcase 3.4: For all $i$, we have $A_i\in\Lambda_2\subset\Lambda_1$ and $\Pi\subset\Omega$. We see that $\lvert \pi_0 \rvert < \lvert \pi \rvert$. By the induction hypothesis for $\Lambda_1$, $\Lambda_2$, $\Omega$ and $\pi_0$ the sequent $\Box(\Lambda_1^\ast),\Lambda_2^\ast,\boxtimes\Omega,\boxtimes \Pi \Rightarrow \boxtimes(A_1,\ldots,A_n)$ is provable in $\mathsf{Go_{Seq}}$. Let $\Psi=\Box(\Lambda_1^\ast),\Lambda_2^\ast,\boxtimes\Omega$. Then we have
\[
\AXC{$\mathsf{Ax}$}
\noLine
\UIC{$\Psi,\boxtimes \Pi, \Box A_1\Rightarrow \boxtimes(A_1,\ldots,A_n)$}
\AXC{$\Psi,\boxtimes \Pi \Rightarrow \boxtimes(A_1,\ldots,A_n)$}
\LeftLabel{$\to_L$}
\BIC{$\Box(\Lambda_1^\ast),\Lambda_2^\ast, A_1\to\Box A_1, \boxtimes\Omega,\boxtimes \Pi \Rightarrow \Box A_1,\boxtimes(A_2,\ldots,A_n)$}
\LeftLabel{$\mathsf{ctr}$}
\UIC{$\Box(\Lambda_1^\ast),\Lambda_2^\ast,\boxtimes\Omega, \boxtimes \Pi \Rightarrow \Box A_1,\boxtimes(A_2,\ldots,A_n)$}
\noLine
\UIC{$\vdots$}
\noLine
\UIC{$\Box(\Lambda_1^\ast),\Lambda_2^\ast, A_n\to\Box A_n,\boxtimes\Omega, \boxtimes \Pi \Rightarrow \Box A_1,\ldots,\Box A_n$}
\LeftLabel{$\mathsf{ctr}$}
\UIC{$\Box(\Lambda_1^\ast),\Lambda_2^\ast,\Box\Omega,\Omega\backslash\Pi,\Pi,\Box \Pi,\Pi\Rightarrow \Box A_1,\ldots,\Box A_n$}
\LeftLabel{$\mathsf{ctr}$}
\UIC{$\Box(\Lambda_1^\ast),\Lambda_2^\ast,\boxtimes\Omega,\Box \Pi\Rightarrow \Box A_1,\ldots,\Box A_n$}
\LeftLabel{$\mathsf{weak}$}
\RightLabel{.}
\UIC{$\Box(\Lambda_1^\ast),\Lambda_2^\ast,\boxtimes\Omega,\Phi,\Box \Pi\Rightarrow \Box A_1,\ldots,\Box A_n,\Sigma$}
\DisplayProof
\]
\end{proof}
From Lemma \ref{translation} we immediately obtain the following theorem.
\begin{theorem}\label{inftoseq}
If $\mathsf{Go_\infty}\vdash \Gamma\Rightarrow\Delta$, then $\mathsf{Go_{Seq}} \vdash \Gamma\Rightarrow\Delta$.
\end{theorem}

\section{Ultrametric spaces}
\label{SecUlt}
In this section we define ultrametrics on some spaces concerning $\infty$-proofs. For some basic notions of the theory of ultrametric spaces (cf. \cite{Shor}) and proofs please see Appendix. 

Consider the set $\mathcal P$ of all $\infty$-proofs of the system $\mathsf{Go}_{\infty} +\mathsf{cut} $. An ultrametric $d_{\mathcal P} \colon {\mathcal P} \times {\mathcal P} \to [0,1]$ on $\mathcal P$ is defined by
\[d_{\mathcal P}(\pi,\tau) = 
\inf\{\frac{1}{2^n} \mid  \pi \sim_n \tau\}.
\]

\begin{prop}\label{ComplP}
$(\mathcal{P},d_{\mathcal P})$ is a (spherically) complete ultrametric space.
\end{prop}

For $m\in \mathbb N$, let $\mathcal F_m$ denote the set of all non-expansive functions from $\mathcal P^m$ to $\mathcal P$. 
For $\mathsf{a,b}\in \mathcal F_m$, we write 
$\mathsf a\sim_{n,k}\mathsf b$ if $\mathsf a(\vec{\pi})\sim_n\mathsf b(\vec{\pi})$ for any $\vec{\pi}\in\mathcal P^m$ and, in addition, $\mathsf a(\vec{\pi}) \sim_{n+1}\mathsf b(\vec{\pi})$ whenever $\Sigma_{i=1}^m\lvert\pi_i\rvert< k$.\footnote{This definition is inspired by \cite[Subsection 2.1]{GiMi}.}
An ultrametric $l_m$ on $\mathcal F_m$ is defined by 
\[l_m(\mathsf a,\mathsf b)=\frac{1}{2}\inf\{\frac{1}{2^n}+\frac{1}{2^{n+k}}\mid \mathsf a\sim_{n,k}\mathsf b\}.\]

\begin{prop}\label{SphCom}
$(\mathcal F_m, l_m)$ is a spherically complete ultrametric space.
\end{prop}

Notice that any operator $\mathsf U\colon\mathcal F_m\to \mathcal F_m$ is strictly contractive if and only if for any $\mathsf a,\mathsf b\in\mathcal F_m$, and any $n,k\in\mathbb N$ we have 
$$\mathsf a\sim_{n,k}\mathsf b\Rightarrow \mathsf U(\mathsf a)\sim_{n,k+1}\mathsf U(\mathsf b).$$

Now we state a generalization of the Banach's fixed-point theorem for ultrametric spaces that will be used in the next sections.

\begin{thm}[Prie\ss-Crampe \cite{PrCr}]\label{fixpoint}
Let $(M,d)$ be a non-empty spherically complete ultrametric space. Then every strictly contractive mapping $f\colon M\to M$ has a unique fixed-point. 
\end{thm}

\section{Admissible Rules and Mappings}

In this section, for the system $\mathsf{Go}_{\infty} +\mathsf{cut}$, we state admissibility of auxiliary inference rules, which will be used in the proof of the  cut-elimination theorem. 

Recall that the set $\mathcal P$ of all $\infty$-proofs of the system $\mathsf{Go}_{\infty} +\mathsf{cut} $ can be considered as an ultrametric space with the metric $d_{\mathcal P}$.

By $\mathcal{P}_n$ we denote the set of all $\infty$-proofs that do not contain applications of the cut rule in their $n$-fragments. We also set $\mathcal{P}_0= \mathcal{P}$.

A mapping $\mathsf u:\mathcal P^m\to \mathcal P$ is called \emph{adequate} if for any $n\in\mathbb N$ we have $\mathsf u(\pi_1,\ldots,\pi_n)\in\mathcal P_n$, whenever $\pi_i\in\mathcal P_n$ for all $i\leqslant n$.

In $\mathsf{Go}_{\infty} +\mathsf{cut}$, we call a single-premise inference rule \emph{strongly admissible} if there is a non-expansive adequate mapping $\mathsf{u}\colon\mathcal{P} \to \mathcal{P}$ that maps any $\infty$-proof of the premise of the rule to an $\infty$-proof of the conclusion. The mapping $\mathsf{u}$ must also satisfy one additional condition: $\lvert \mathsf{u}(\pi)\rvert \leqslant \lvert \pi \rvert$ for any $\pi \in \mathcal{P}$.

In the following lemmata, non-expansive mappings are defined in a standard way by induction on the local heights of $\infty$-proofs for the premises. So we omit further details.

\begin{lem}\label{strongweakening}
For any finite multisets of formulas $\Pi$ and $\Sigma$, the inference rule
\begin{gather*}
\AXC{$\Gamma\Rightarrow\Delta$}
\LeftLabel{$\mathsf{wk}_{\Pi; \Sigma}$}
\UIC{$\Pi,\Gamma\Rightarrow\Delta,\Sigma$}
\DisplayProof
\end{gather*}
is strongly admissible in $\mathsf{Go}_{\infty} +\mathsf{cut}$. %Its corresponding mapping we denote by $\mathsf{w}_{\Pi, \Sigma}$
\end{lem}

\begin{lem}\label{inversion}
For any formulas $A$ and $B$, the rules
\begin{gather*}
\AXC{$\Gamma , A \rightarrow B \Rightarrow  \Delta$}
\LeftLabel{$\mathsf{li}_{A \to B}$}
\UIC{$\Gamma ,B \Rightarrow  \Delta$}
\DisplayProof\qquad
\AXC{$\Gamma , A \rightarrow B \Rightarrow  \Delta$}
\LeftLabel{$\mathsf{ri}_{A \to B}$}
\UIC{$\Gamma  \Rightarrow  A, \Delta$}
\DisplayProof\\\\
\AXC{$\Gamma  \Rightarrow   A \rightarrow B, \Delta$}
\LeftLabel{$\mathsf{i}_{A \to B}$}
\UIC{$\Gamma ,A \Rightarrow B, \Delta$}
\DisplayProof\qquad
\AXC{$\Gamma  \Rightarrow   \bot, \Delta$}
\LeftLabel{$\mathsf{i}_{\bot}$}
\UIC{$\Gamma  \Rightarrow \Delta$}
\DisplayProof
\end{gather*}
are strongly admissible in $\mathsf{Go}_{\infty} +\mathsf{cut}$.
\end{lem}

\begin{lem}\label{weakcontraction}
For any atomic proposition $p$, the rules
\begin{gather*}
\AXC{$\Gamma , p,p \Rightarrow  \Delta$}
\LeftLabel{$\mathsf{acl}_{p}$}
\UIC{$\Gamma ,p \Rightarrow  \Delta$}
\DisplayProof\qquad
\AXC{$\Gamma \Rightarrow p,p, \Delta$}
\LeftLabel{$\mathsf{acr}_{p}$}
\UIC{$\Gamma \Rightarrow p, \Delta$}
\DisplayProof
\end{gather*}
are strongly admissible in $\mathsf{Go}_{\infty} +\mathsf{cut}$.
\end{lem}

Let us also define the mapping $\mathsf{clip}\colon \mathcal P\to\mathcal P$. Consider an $\infty$-proof $\pi$. If the last rule application in $\pi$ is not of the rule ($\Box$) then we put $\mathsf{clip}(\pi)=\pi$. If the $\infty$-proof $\pi$ has the form
\[
\AXC{$\pi_0$}
\noLine
\UIC{$\boxtimes \Pi \Rightarrow \boxtimes(A_1,\ldots,A_n)$}
\AXC{$\pi_1$}
\noLine
\UIC{$\boxtimes \Pi\Rightarrow A_1$}
\AXC{$\ldots$}
\AXC{$\pi_n$}
\noLine
\UIC{$\boxtimes \Pi\Rightarrow A_n$}
\QuaternaryInfC{$\Gamma, \Box \Pi\Rightarrow \Box A_1,\ldots,\Box A_n, \Delta$}
\DisplayProof
\]
we define $\mathsf{clip}(\pi)=\pi$ to be
\[
\AXC{$\pi_0$}
\noLine
\UIC{$\boxtimes \Pi \Rightarrow \boxtimes(A_1,\ldots,A_n)$}
\AXC{$\pi_1$}
\noLine
\UIC{$\boxtimes \Pi\Rightarrow A_1$}
\AXC{$\ldots$}
\AXC{$\pi_n$}
\noLine
\UIC{$\boxtimes \Pi\Rightarrow A_n$}
\RightLabel{.}
\QuaternaryInfC{$\Box \Pi\Rightarrow \Box A_1,\ldots,\Box A_n$}
\DisplayProof
\]
Clearly this mapping is non-expansive, adequate, and $\lvert \mathsf{clip}(\pi)\rvert \leqslant \lvert \pi \rvert$ for any $\pi \in \mathcal{P}$.

\section{Cut elimination}
\label{SecCut}

In this section, we construct a continuous function from $\mathcal{P}$ to $\mathcal{P}$, which maps any $\infty$-proof of the system $\mathsf{Go}_{\infty} +\mathsf{cut}$ to a cut-free $\infty$-proof of the same sequent. 

Let us call a pair of $\infty$-proofs $(\pi,\tau)$ a \emph{cut pair} if $\pi$ is an $\infty$-proof of the sequent $\Gamma\Rightarrow\Delta,A$ and $\tau$ is an $\infty$-proof of the sequent $A, \Gamma\Rightarrow \Delta$ for some $\Gamma, \Delta$, and $A$. For a cut pair $(\pi,\tau)$, we call the sequent $\Gamma\Rightarrow \Delta$ its \emph{cut result} and the formula $A$ its cut formula.

For a modal formula $A$, a non-expansive mapping $\mathsf{u}$ from $\mathcal P \times \mathcal P$ to $\mathcal P$ is called \emph{$A$-removing} if it maps every cut pair $(\pi,\tau)$ with the cut formula $A$ to an $\infty$-proof of its cut result.
By $\mathcal R_A$, let us denote the set of all $A$-removing mappings.

\begin{lem}\label{Comp R_A}
For each formula $A$, the pair $(\mathcal R_A,l_2)$ is a non-empty spherically complete ultrametric space.

\end{lem}
\begin{proof}
The proof of spherical completeness of the space $(\mathcal R_A,l_2)$ is analoguos to the proof of the spherical completeness of $(\mathcal F_m,L_m)$.

We only need to check that the set $\mathcal R_A$ is non-empty. Consider the mapping $\mathsf u_{cut}:\mathcal P^2\to\mathcal P$ that is defined as follows. For a cut pair $(\pi,\tau)$ with the cut formula $A$, it joins the $\infty$-proofs $\pi$ and $\tau$ with an appropriate instance of the rule $(\mathsf{cut})$. For all other pairs, the mapping $\mathsf u_{cut}$ returns the first argument.

Clearly, $\mathsf u_{cut}$ is non-expansive and therefore lies in $\mathcal R_A$.
\end{proof}

In what follows, we use nonexpansive adequate mappings $ \mathsf{wk}_{\Pi; \Sigma}$, $\mathsf{li}_{A\to B}$, $\mathsf{ri}_{A\to B}$, $\mathsf{i}_{A\to B}$, $\mathsf{i}_\bot$, $\mathsf{acl}_{p}$, $\mathsf{acr}_{p}$ from Lemma \ref{strongweakening}, Lemma \ref{inversion}, and Lemma \ref{weakcontraction}.

\begin{lem}\label{repadeq}
For any atomic proposition $p$, there exists an adequate $p$-removing mapping $\mathsf {re}_p$.
\end{lem}
\begin{proof}

Assume we have two $\infty$-proofs $\pi$ and $\tau$. If the pair $(\pi,\tau)$ is not a cut pair or is a cut pair with the cut formula being not $p$, then we put $\mathsf{re}_{p}(\pi,\tau)=\pi$. 
Otherwise, we define $\mathsf{re}_{p}(\pi,\tau)$ by induction on $\lvert \pi\rvert$. Let the cut result of  the pair $(\pi,\tau)$ be $\Gamma\Rightarrow \Delta$.

If $\lvert \pi\rvert=0$, then $\Gamma\Rightarrow \Delta, p$ is an initial sequent. Suppose that $\Gamma\Rightarrow \Delta$ is also an initial sequent. Then $\mathsf{re}_{p}(\pi,\tau)$ is defined as the $\infty$-proof consisting only of the sequent $\Gamma\Rightarrow \Delta$. If $\Gamma\Rightarrow \Delta$ is not an initial sequent, then $\Gamma$ has the form $p,\Phi$, and $\tau$ is an $\infty$-proof of the sequent $p,p,\Phi \Rightarrow \Delta$. Applying the non-expansive adequate mapping $\mathsf{acl}_p$ from Lemma \ref{weakcontraction}, we put $\mathsf{re}_{p}(\pi,\tau) := \mathsf{acl}_p (\tau)$.   

Now suppose that $\lvert \pi \rvert >0$. We define $\mathsf{re}_{p}(\pi,\tau)$ according to the last application of an inference rule in $\pi$:  
\begin{align*}
%\[
\left(\mathsf{\to_{R(B\to C)}}(\pi_0)
\AXC{$\pi_0$}
\noLine
\UIC{$\Gamma,B\Rightarrow C,\Sigma,p$}
\LeftLabel{$\mathsf{\to_R}$}
\UIC{$\Gamma\Rightarrow B\to C,\Sigma,p$}
%\DisplayProof
,\tau
\right)
&
\longmapsto
\AXC{$\mathsf{re}_p(\pi_0, \mathsf{i}_{B \to C}(\tau))$}
\noLine
\UIC{$\Gamma,B\Rightarrow C,\Sigma$}
\LeftLabel{$\mathsf{\to_R}$}
\RightLabel{ ,}
\UIC{$\Gamma\Rightarrow B\to C,\Sigma$}
\DisplayProof
%\]
\\\\
%\[
\left(\mathsf{\to_{L(B\to C)}}(\pi_0,\pi_1)
\AXC{$\pi_0$}
\noLine
\UIC{$\Sigma, C\Rightarrow \Delta, p$}
\AXC{$\pi_1$}
\noLine
\UIC{$\Sigma \Rightarrow B,\Delta, p$}
\LeftLabel{$\mathsf{\to_L}$}
\BIC{$\Sigma, B\to C\Rightarrow \Delta, p$}
%\DisplayProof
,\tau
\right)
&
\longmapsto
\AXC{$\mathsf{re}_p (\pi_0, \mathsf{li}_{B\to C} (\tau))$}
\noLine
\UIC{$\Sigma, C\Rightarrow \Delta$}
\AXC{$\mathsf{re}_p (\pi_1, \mathsf{ri}_{B\to C} (\tau))$}
\noLine
\UIC{$\Sigma \Rightarrow B,\Delta$}
%\LeftLabel{$\mathsf{\to_L}$}
\RightLabel{ ,}
\BIC{$\Sigma, B\to C\Rightarrow \Delta$}
\DisplayProof
%\]
\\\\
%\[
\left(\mathsf{cut_B}(\pi_0,\pi_1)
\AXC{$\pi_0$}
\noLine
\UIC{$\Gamma\Rightarrow B,\Delta, p$}
\AXC{$\pi_1$}
\noLine
\UIC{$\Gamma,B \Rightarrow \Delta, p$}
%\LeftLabel{$\mathsf{cut}$}
\BIC{$\Gamma\Rightarrow \Delta, p$}
%\DisplayProof
,\tau
\right)
&
\longmapsto
\AXC{$\mathsf{re}_p (\pi_0, \mathsf{wk}_{\emptyset;B} (\tau))$}
\noLine
\UIC{$\Gamma\Rightarrow B,\Delta$}
\AXC{$\mathsf{re}_p (\pi_1, \mathsf{wk}_{B;\emptyset} (\tau))$}
\noLine
\UIC{$\Gamma, B\Rightarrow \Delta$}
\LeftLabel{$\mathsf{cut}$}
\RightLabel{ ,}
\BIC{$\Gamma\Rightarrow \Delta$}
\DisplayProof
%\]
%\end{align*}
\\\\
%\[
\left(\mathsf{\Box_{\Phi;\Sigma,p}(\pi_0,\pi_1,\ldots,\pi_n)}
\AXC{$\pi_0$}
\noLine
\UIC{$\boxtimes \Pi \Rightarrow \boxtimes(A_1,\ldots,A_n)$}
\AXC{$\ldots$}
\LeftLabel{$\mathsf{\Box}$}
\BIC{$\Phi, \Box \Pi \Rightarrow \Box A_1,\ldots,\Box A_n, \Sigma, p$}
%\DisplayProof
,\tau\right)
&
\longmapsto \mathsf{\Box_{\Phi;\Sigma}(\pi_0,\pi_1,\ldots,\pi_n)}
\AXC{$\pi_0$}
\noLine
\UIC{$\boxtimes \Pi \Rightarrow \boxtimes(A_1,\ldots,A_n)$}
\AXC{$\ldots$}
\LeftLabel{$\mathsf{\Box}$}
\BIC{$\Phi, \Box \Pi \Rightarrow \Box A_1,\ldots,\Box A_n, \Sigma$}
%\DisplayProof
%\]
\end{align*}
The mapping $\mathsf{re}_p$ is well defined, adequate and non-expansive. 
\end{proof}
\begin{lem}\label{reboxadeq}
Given an adequate $B$-removing mapping $\mathsf{re}_B$, there exists an adequate $\Box B$-removing mapping $\mathsf{re}_{\Box B}$.
\end{lem}
\begin{proof}
Assume we have an adequate $B$-removing mapping $\mathsf{re}_B$. The required $\Box B$-removing mapping $\mathsf{re}_{\Box B}$ is obtained as the fixed-point of a contractive operator $\mathsf G_{\Box B} \colon \mathcal R_{\Box B} \to \mathcal R_{\Box B}$. 

For a mapping $\mathsf u\in \mathcal R_{\Box B}$ and a pair of $\infty$-proofs $(\pi,\tau)$, the $\infty$-proof $\mathsf G_{\Box B}(\mathsf u)(\pi,\tau)$ is defined as follows. 
If $(\pi,\tau)$ is not a cut pair or a cut pair with 
the cut formula being not $\Box B$, then we put $\mathsf G_{\Box B}(\mathsf u)(\pi,\tau)=\pi$. 

Now let $(\pi,\tau)$ be a cut pair with 
the cut formula $\Box B$ and the cut result $\Gamma\Rightarrow \Delta$.
If $\lvert \pi\rvert=0$ or $\lvert \tau \rvert=0$, then $\Gamma\Rightarrow \Delta$ is an initial sequent. In this case, we define $\mathsf G_{\Box B}(\mathsf u)(\pi,\tau)$ as the $\infty$-proof consisting only of the sequent $\Gamma\Rightarrow \Delta$. 

Suppose that $\lvert \pi\rvert>0$ and $\lvert \tau \rvert>0$. We define $\mathsf G_{\Box B}(\mathsf u)(\pi,\tau)$ according to the last application of an inference rule in $\pi$:\\
\begin{align*}
%\[
\left(\mathsf{\to_{R(C\to D)}}(\pi_0)
\AXC{$\pi_0$}
\noLine
%\UIC{$\vdots$}
%\noLine
\UIC{$\Gamma,C\Rightarrow D,\Sigma,\Box B$}
\LeftLabel{$\mathsf{\to_R}$}
\UIC{$\Gamma\Rightarrow C\to D,\Sigma,\Box B$}
%\DisplayProof
,\tau
\right)
&
\longmapsto
\AXC{$\mathsf{u}(\pi_0, \mathsf{i}_{C \to D}(\tau))$}
\noLine
%\UIC{$\vdots$}
%\noLine
\UIC{$\Gamma,C\Rightarrow D,\Sigma$}
\LeftLabel{$\mathsf{\to_R}$}
\RightLabel{ ,}
\UIC{$\Gamma\Rightarrow C\to D,\Sigma$}
\DisplayProof
%\]
\\\\
%\[
\left(\mathsf{\to_{L(C\to D)}}(\pi_0,\pi_1)
\AXC{$\pi_0$}
\noLine
%\UIC{$\vdots$}
%\noLine
\UIC{$\Sigma, D\Rightarrow \Delta, \Box B$}
\AXC{$\pi_1$}
\noLine
%\UIC{$\vdots$}
%\noLine
\UIC{$\Sigma \Rightarrow C,\Delta, \Box B$}
\LeftLabel{$\mathsf{\to_L}$}
\BIC{$\Sigma, C\to D\Rightarrow \Delta, \Box B$}
%\DisplayProof
,\tau
\right)
&
\longmapsto
\AXC{$\mathsf{u} (\pi_0, \mathsf{li}_{C\to D} (\tau))$}
\noLine
%\UIC{$\vdots$}
%\noLine
\UIC{$\Sigma, D\Rightarrow \Delta$}
\AXC{$\mathsf{u} (\pi_1, \mathsf{ri}_{C\to D} (\tau))$}
\noLine
%\UIC{$\vdots$}
%\noLine
\UIC{$\Sigma \Rightarrow C,\Delta$}
\LeftLabel{$\mathsf{\to_L}$}
\RightLabel{ ,}
\BIC{$\Sigma, C\to D\Rightarrow \Delta$}
\DisplayProof
%\]
%\end{align*}
\\\\
%\begin{align*}
%\[
\left(\mathsf{cut_C}(\pi_0,\pi_1)
\AXC{$\pi_0$}
\noLine
\UIC{$\Gamma\Rightarrow C,\Delta, \Box B$}
\AXC{$\pi_1$}
\noLine
\UIC{$\Gamma,C \Rightarrow \Delta, \Box B$}
\LeftLabel{$\mathsf{cut}$}
\BIC{$\Gamma\Rightarrow \Delta, \Box B$}
%\DisplayProof
,\tau
\right)
&
\longmapsto
\AXC{$\mathsf{u} (\pi_0, \mathsf{wk}_{\emptyset;C} (\tau))$}
\noLine
\UIC{$\Gamma\Rightarrow C,\Delta$}
\AXC{$\mathsf{u} (\pi_1, \mathsf{wk}_{C;\emptyset} (\tau))$}
\noLine
\UIC{$\Gamma, C\Rightarrow \Delta$}
\LeftLabel{$\mathsf{cut}$}
\RightLabel{ ,}
\BIC{$\Gamma\Rightarrow \Delta$}
\DisplayProof
%\]
\end{align*}
%\\\\
\[
\left(\mathsf{\Box_{\Phi;\Sigma,\Box B}(\pi_0,\pi_1,\ldots,\pi_n)}
\AXC{$\pi_0$}
\noLine
\UIC{$\boxtimes \Pi \Rightarrow \boxtimes(A_1,\ldots,A_n)$}
\AXC{$\ldots$}
\LeftLabel{$\mathsf{\Box}$}
\BIC{$\Phi, \Box \Pi \Rightarrow \Box A_1,\ldots,\Box A_n, \Sigma, \Box B$}
%\DisplayProof
,\tau\right)
%&
\longmapsto \mathsf{\Box_{\Phi;\Sigma}(\pi_0,\pi_1,\ldots,\pi_n)}
\AXC{$\pi_0$}
\noLine
\UIC{$\boxtimes \Pi \Rightarrow \boxtimes(A_1,\ldots,A_n)$}
\AXC{$\ldots$}
\LeftLabel{$\mathsf{\Box}$}
\BIC{$\Phi, \Box \Pi \Rightarrow \Box A_1,\ldots,\Box A_n, \Sigma$}
%\DisplayProof
\]
%\end{align*}

Consider the case when $\pi$ has the form 
\[
\AXC{$\pi_0$}
\noLine
\UIC{$\boxtimes \Pi \Rightarrow \boxtimes (B, A_1,\ldots,A_n)$}
\AXC{$\pi_B$}
\noLine
\UIC{$\boxtimes \Pi \Rightarrow B$}
\AXC{$\pi_1$}
\noLine
\UIC{$\boxtimes \Pi \Rightarrow A_1\quad\ldots$}
%\AXC{$\ldots$}
\AXC{$\boxtimes \Pi \Rightarrow \boxtimes A_n$}
\LeftLabel{$\mathsf{\Box}$}
\RightLabel{ .}
\QuaternaryInfC{$\Phi, \Box \Pi \Rightarrow \Box B,\Box A_1,\ldots,\Box A_n, \Sigma$}
\DisplayProof
\]
We define $\mathsf G_{\Box B}(\mathsf u)(\pi,\tau)$ according to the last application of an inference rule in $\tau$. 

\begin{align*}
\left(\pi,\mathsf{\to_{R(C\to D)}}(\tau_0)
\AXC{$\tau_0$}
\noLine
%\UIC{$\vdots$}
%\noLine
\UIC{$\Gamma,\Box B,C\Rightarrow D,\Sigma$}
\LeftLabel{$\mathsf{\to_R}$}
\UIC{$\Gamma,\Box B\Rightarrow C\to D,\Sigma$}
%\DisplayProof
\right)
&\longmapsto
\AXC{$\mathsf{u}(\mathsf{i}_{C \to D}(\pi), \tau_0)$}
\noLine
%\UIC{$\vdots$}
%\noLine
\UIC{$\Gamma,C\Rightarrow D,\Sigma$}
\LeftLabel{$\mathsf{\to_R}$}
\RightLabel{ ,}
\UIC{$\Gamma\Rightarrow C\to D,\Sigma$}
\DisplayProof
\\\\
\left(\pi,\mathsf{\to_{L(C\to D)}}(\tau_0,\tau_1)
\AXC{$\tau_0$}
\noLine
%\UIC{$\vdots$}
%\noLine
\UIC{$\Sigma, \Box B, D\Rightarrow \Delta$}
\AXC{$\tau_1$}
\noLine
%\UIC{$\vdots$}
%\noLine
\UIC{$\Sigma, \Box B \Rightarrow C,\Delta$}
\LeftLabel{$\mathsf{\to_L}$}
\BIC{$\Sigma, \Box B, C\to D\Rightarrow \Delta$}
%\DisplayProof
\right)
&\longmapsto
\AXC{$\mathsf{u} (\mathsf{li}_{C\to D}(\pi), \tau_0)$}
\noLine
%\UIC{$\vdots$}
%\noLine
\UIC{$\Sigma, D\Rightarrow \Delta$}
\AXC{$\mathsf{u} ( \mathsf{ri}_{C\to D}(\pi_1), \tau_1)$}
\noLine
%\UIC{$\vdots$}
%\noLine
\UIC{$\Sigma \Rightarrow C,\Delta$}
\LeftLabel{$\mathsf{\to_L}$}
\RightLabel{ ,}
\BIC{$\Sigma, C\to D\Rightarrow \Delta$}
\DisplayProof
%\end{align*}
%\begin{align*}
\\\\
\left(\pi,\mathsf{cut_C}(\tau_0,\tau_1)
\AXC{$\tau_0$}
\noLine
\UIC{$\Gamma, \Box B\Rightarrow C,\Delta$}
\AXC{$\tau_1$}
\noLine
\UIC{$\Gamma, \Box B,C \Rightarrow \Delta$}
\LeftLabel{$\mathsf{cut}$}
\BIC{$\Gamma, \Box B\Rightarrow \Delta$}
%\DisplayProof
\right)
&\longmapsto
\AXC{$\mathsf{u} (\mathsf{wk}_{\emptyset;C}(\pi), \tau_0)$}
\noLine
\UIC{$\Gamma\Rightarrow C,\Delta$}
\AXC{$\mathsf{u} (\mathsf{wk}_{C;\emptyset} (\pi), \tau_1)$}
\noLine
\UIC{$\Gamma, C\Rightarrow \Delta$}
\LeftLabel{$\mathsf{cut}$}
\RightLabel{ ,}
\BIC{$\Gamma\Rightarrow \Delta$}
\DisplayProof
\\\\
(\pi,\mathsf{\Box_{\Phi',\Box B;\Sigma'}(\tau_0,\tau_1,\ldots,\tau_k)})&\longmapsto\mathsf{\Box_{\Phi';\Sigma'}(\tau_0,\tau_1,\ldots,\tau_k)}
\end{align*}
%\[
%\left(\pi,\mathsf{\Box_{\Phi,\Box B;\Sigma}(\tau_0,\pi_1,\ldots,\tau_k)}
%\AXC{$\tau_0$}
%\noLine
%\UIC{$\boxtimes \Lambda \Rightarrow \boxtimes(C_1,\ldots,C_k)$}
%\AXC{$\ldots$}
%\LeftLabel{$\mathsf{\Box}$}
%\BIC{$\Phi, \Box B, \Box \Lambda \Rightarrow \Box C_1,\ldots,\Box C_k, \Sigma$}
%\DisplayProof
%\right)
%\longmapsto\mathsf{\Box_{\Phi;\Sigma}(\tau_0,\pi_1,\ldots,\tau_k)
%\AXC{$\tau_0$}
%\noLine
%\UIC{$\boxtimes \Lambda \Rightarrow \boxtimes(C_1,\ldots,C_k)$}
%\AXC{$\ldots$}
%\LeftLabel{$\mathsf{\Box}$}
%\BIC{$\Phi, \Box \Lambda \Rightarrow \Box C_1,\ldots,\Box C_k, \Sigma$}
%\DisplayProof
%\]
%\end{align*}

It remains to consider the case when $\tau$ has the form
\[
\AXC{$\tau_0$}
\noLine
\UIC{$\boxtimes \Lambda,\boxtimes B \Rightarrow \boxtimes (C_1,\ldots,C_k)$}
\AXC{$\tau_1$}
\noLine
\UIC{$\boxtimes \Lambda,\boxtimes B \Rightarrow C_1$}
\AXC{$\ldots$}
\AXC{$\boxtimes \Lambda,\boxtimes B \Rightarrow \boxtimes C_k$}
\LeftLabel{$\mathsf{\Box}$}
\RightLabel{ .}
\QuaternaryInfC{$\Phi', \Box \Lambda,\Box B \Rightarrow \Box C_1,\ldots,\Box C_k, \Sigma'$}
\DisplayProof
\]

Notice that the sequent $\Gamma\Rightarrow\Delta$, the sequent $\Phi, \Box \Pi \Rightarrow \Box A_1,\ldots,\Box A_n, \Sigma$, and the sequent $\Phi', \Box \Lambda \Rightarrow \Box C_1,\ldots,\Box C_k, \Sigma'$ are the same.

Let $I:=\{i\mid A_i\notin{C_1,\ldots,C_k}\}$ and $J:=\{i\mid C_i\notin{A_1,\ldots,A_n}\}$, let $\pi'_B$ be the followind $\infty$-proof:
\[
\AXC{$\mathsf{wk_{\varnothing,\Box B}}(\pi_B)$}
\noLine
\UIC{$\boxtimes\Pi\Rightarrow \boxtimes B$}
\AXC{$\pi_B$}
\noLine
\UIC{$\boxtimes\Pi\Rightarrow B$}
\LeftLabel{$\Box$}
\RightLabel{,}
\BIC{$\boxtimes\Pi\Rightarrow\Box B$}
\DisplayProof
\]
and condsider the following $\infty$-proofs:
%\begin{gather*}
$$\psi_1:=\mathsf u(\mathsf{wk}_{\Lambda\backslash \Pi,\Box{\Lambda\backslash\Pi};\{\boxtimes C_i\}_{i\in J}}(\pi_0),\mathsf{wk}_{\Pi\cup\Lambda,\Box(\Pi\backslash\Lambda);\{\boxtimes A_i\}_{i\in I},\{C_i\}}(\mathsf{clip}(\tau))),$$
$$\psi_2:=\mathsf u(\mathsf{wk}_{\Pi\cup\Lambda,\Box{\Lambda\backslash\Pi};\{\boxtimes C_i\}_{i\in J},\{A_i\}}(\mathsf{clip}(\pi)),\mathsf{wk}_{\Pi\backslash\Lambda,\Box(\Pi\backslash\Lambda);\{\boxtimes A_i\}_{i\in I}}(\tau_0)),$$
$$\phi_j:=\mathsf u(\mathsf{wk}_{\boxtimes(\Lambda\backslash\Pi);C_j}(\pi'_B),\mathsf{re_B}(\mathsf{wk}_{\boxtimes(\Lambda\backslash\Pi,\Box B;\varnothing)}(\pi_B),\mathsf{wk}_{\boxtimes(\Pi\backslash\Lambda;\varnothing)}(\tau_j))).$$
%\end{gather*}
We define $\mathsf G_{\Box B}(\mathsf u)(\pi,\tau)$ as
\[
\AXC{$\mathsf{re_B}(\psi_1,\psi_2)$}
\noLine
\UIC{$\boxtimes(\Pi\cup\Lambda)\Rightarrow\boxtimes {A_i},\{\boxtimes C_j\}_{j\in J}$}
\AXC{$(\mathsf{wk}_{\boxtimes(\Lambda\backslash\Pi)}(\pi_i))$}
\noLine
\UIC{$(\boxtimes(\Pi\cup\Lambda)\Rightarrow A_i)$}
\AXC{$\phi_j$}
\noLine
\UIC{$(\boxtimes(\Pi\cup\Lambda)\Rightarrow C_j))_{j\in J}$}
\LeftLabel{$\Box$}
\RightLabel{.}
\TIC{$\Gamma\Rightarrow\Delta$}
\DisplayProof
\]

Now the operator $\mathsf {G_{\Box B}}$ is well-defined.
By the case analysis according to the definition of $\mathsf {G_{\Box B}}$, we see that $\mathsf {G_{\Box B}} (\mathsf u) $ is non-expansive and belongs to $\mathcal R_{\Box B}$ whenever $\mathsf u \in \mathcal R_{\Box B}$.

We claim that $\mathsf {G_{\Box B}}$ is contractive. It sufficient to check that for any $\mathsf u, \mathsf v\in \mathcal R_{\Box B}$ and any $n,k\in\mathbb N$ we have 
\[\mathsf u\sim_{n,k}\mathsf v \Rightarrow \mathsf {G_{\Box B}(u)}\sim_{n,k+1}\mathsf{G_{\Box B}(v)},\]
which we prove by case analysis.

%Let us check that the operator $\mathsf {G_{\Box B}}$ is contractive. It is sufficient to prove that if for some $\mathsf u, \mathsf v\in \mathcal R_{\Box B}$ and $n,k\in\mathbb N$ we have $\mathsf u\sim_{n,k}\mathsf v$, then $\mathsf {G_{\Box B}(u)}\sim_{n,k+1}\mathsf{G_{\Box B}(v)}$. Let $\mathsf u\sim_{n,k}\mathsf v$. Consider a pair of $\infty$-proofs $(\pi,\tau)$. 

Now we define the required $\Box B$-removing mapping $\mathsf{re}_{\Box B}$ as the fixed-point of the the operator $\mathsf G_{\Box B} \colon \mathcal R_{\Box B} \to \mathcal R_{\Box B}$, which exists by Lemma \ref{Comp R_A} and Theorem \ref{fixpoint}.

%, the operator $\mathsf G_{\Box B} \colon \mathcal R_{\Box B} \to \mathcal R_{\Box B}$ has a fixed-point, which we denote by $\mathsf{re}_{\Box B}$.
%We proved that the mapping $\mathsf G_{\Box B}$ is a contracting operator on a spherically complete ultrametric space $\mathcal R_{\Box B}$. Therefore by Theorem \ref{fixpoint} it has a fixed-point in this set. We denote this fixed-point  $\mathsf{re}_{\Box B}$.

It remains to check that the mapping $\mathsf{re}_{\Box B}$ is adequate. For some $n,k\in\mathbb N$, let us call a mapping $\mathsf u\in \mathcal R_{\Box B}$ $(n,k)$-adequate if it satisfies the following two conditions: $\mathsf u(\pi,\tau)\in\mathcal P_i$ for any $i\leqslant n$ and any $\pi, \tau \in \mathcal P_i$; $\mathsf u(\pi,\tau)\in\mathcal P_{n+1}$ whenever $\pi, \tau \in \mathcal P_{n+1}$ and $\lvert \pi\rvert + \lvert \tau \rvert <k$.

By case analysis we establish that $\mathsf G_{\Box B}(\mathsf u)$ is $(n,k+1)$-adequate for any $(n,k)$-adequate $\mathsf u\in\mathcal R_{\Box B}$. Notice that if a mapping $\mathsf{u} $ is $(n,k)$-adequate for all $k\in\mathbb N$, then it is also $(n+1,0)$-adequate. Now by induction on $n$ with a subinduction on $k$, we immediately obtain that the mapping $\mathsf {re}_{\Box B}$, which is a fixed-point of $\mathsf G_{\Box B}$, is $(n,k)$-adequate for all $n,k\in\mathbb N$. Therefore the mapping $\mathsf {re}_{\Box B}$ is adequate. 
%Since $\mathsf G_{\Box B}(\mathsf {re}_{\Box B})=\mathsf{re}_{\Box B}$, the mapping $\mathsf{re}_{\Box B}$ is $(n,k)$-adequate for all $n,k\in\mathbb N$ and therefore adequate.

%For some $n,k\in\mathbb N$, let us call a mapping $\mathsf u\in \mathcal R_{\Box B}$ $(n,k)$-adequate if $\mathsf u(\pi_1,\ldots,\pi_m)\in\mathcal P_n$ whenever $\pi_i\in \mathcal P_n$ for all $i\leqslant m$ and $\mathsf u(\pi_1,\ldots,\pi_m)\in\mathcal P_{n+1}$ if additionally $\sum_{i=1}^m\lvert \pi_i\rvert<k$. By case analysis we can show that for any $(n,k)$-adequate $\mathsf a\in\mathcal R_{\Box B}$ the mapping $\mathsf G_{\Box B}(\mathsf a)$ is $(n,k+1)$-adequate. If a mapping is $(n,k)$-adequate for all $k\in\mathbb N$, it is also $(n+1,0)$-adequate. Since $\mathsf G_{\Box B}(\mathsf {re}_{\Box B})=\mathsf{re}_{\Box B}$, the mapping $\mathsf{re}_{\Box B}$ is $n,k$-adequate for all $n,k\in\mathbb N$ and therefore adequate.

\end{proof}

\begin{lem}\label{reabadeq}
For any formula $A$, there exists an adequate $A$-removing mapping $\mathsf{re}_A$.
%Given an adequate $B$-removing operator $\mathsf{re}_B$ and an adequate $C$-removing operator $\mathsf{re}_C$, there exists an adequate $(B\to C)$-removing operator $\mathsf{re}_{(B\to C)}$.
\end{lem}
\begin{proof}

We define $\mathsf{re}_A$ by induction on the structure of the formula $A $.

%We will define $\mathcal R_{A}$ inductively by $\lvert A\rvert$ (with subinduction on $h(\pi_1)+h(\pi_2)$ in the case when $A=\Box A'$).

%If $\lvert A\rvert=1$, then either $A=p$, in which case $\mathcal R_{A}(\pi_1,\pi_2)$ is already defined in lemma \ref{smallcut}, or $A=\bot$, in which case we put $\mathcal R_A(\pi_1,\pi_2)=\mathcal D(\pi_1)$.

Case 1: $A$ has the form $p$. In this case, $\mathsf{re}_{p} $ is defined by Lemma \ref{repadeq}.
 
Case 2: $A$ has the form $\bot$. Then we put $\mathsf{re}_{\bot}(\pi,\tau):= \mathsf{i}_\bot (\pi)$, where $\mathsf{i}_\bot$ is a non-expansive adequate mapping from Lemma \ref{inversion}.

Case 3: $A$ has the form $B\to C$. Then we put    
\[\mathsf{re}_{B \to C}(\pi,\tau):=\mathsf{re}_C(\mathsf{re}_B(\mathsf{wk}_{\emptyset, C}(\mathsf{ri}_{B\to C}(\tau)),\mathsf{i}_{B\to C}(\pi)),\mathsf{li}_{B\to C}(\tau))\,\]
where $\mathsf{ri}_{B \to C}$, $\mathsf{i}_{B \to C}$, $\mathsf{li}_{B \to C}$ are non-expansive adeqate mappings from Lemma \ref{inversion} and $\mathsf{wk}_{\emptyset, C}$ is a non-expansive adequate mapping from Lemma \ref{strongweakening}.

Case 4: $A$ has the form $\Box B$. By the induction hypothesis, there is an adequate $B$-removing mapping $\mathsf{re}_B$. The required $\Box B$-removing mapping $\mathsf{re}_{\Box B}$ exists by Lemma \ref{reboxadeq}.
\end{proof}

A mapping $\mathsf{u}\colon \mathcal P \to\mathcal P$ is called \emph{root-preserving} if it maps $\infty$-proofs to $\infty$-proofs of the same sequents. Let $\mathcal{T}$ denote the set of all root-preserving non-expansive mappings from $\mathcal P$ to $\mathcal P$. 

\begin{lem}\label{Comp T}
The pair $(\mathcal T, l_1)$ is a non-empty spherically complete ultrametric space.
\end{lem}

\begin{thm}[cut-elimination]
\label{infcuttoinf}
If $\mathsf{Grz_\infty} +\mathsf{cut}\vdash\Gamma\Rightarrow\Delta$, then $\mathsf{Grz_\infty}\vdash\Gamma\Rightarrow\Delta$.
\end{thm}
\begin{proof}
We obtain the required cut-elimination mapping $\mathsf{ce}$ as the fixed-point of a contractive operator $\mathsf F \colon \mathcal T \to \mathcal T$. 

For a mapping $\mathsf u\in \mathcal T$ and an $\infty$-proof $\pi$, the $\infty$-proof $\mathsf F(\mathsf u)(\pi)$ is defined as follows. If $\lvert \pi\rvert=0$, then we define $\mathsf F(\mathsf u)(\pi)$ to be $\pi$.

Otherwise, we define $\mathsf F(\mathsf u)(\pi)$ according to the last application of an inference rule in $\pi$:
\begin{align*}
\mathsf{\to_{R(A\to B)}}(\pi_0) &\longmapsto \mathsf{\to_{R(A\to B)}}(\mathsf u(\pi_0))\\
\mathsf{\to_{L(A\to B)}}(\pi_0,\pi_1) &\longmapsto \mathsf{\to_{L(A\to B)}}(\mathsf u(\pi_0),\mathsf u(\pi_1))\\
\mathsf{\Box_{A_1,\ldots,A_n}}(\pi_0,\pi_1,\ldots,\pi_n)&\longmapsto\mathsf{\Box_{A_1,\ldots,A_n}}(\mathsf u(\pi_0),\mathsf u(\pi_1),\ldots,\mathsf u(\pi_n))\\
\mathsf{cut_A}(\pi_0,\pi_1) &\longmapsto \mathsf{re_A}(\pi_0,\pi_1) 
\end{align*}
Now the operator $\mathsf F$ is well-defined.
By the case analysis according to the definition of $\mathsf F$, we see that $\mathsf F (\mathsf u) $ is non-expansive and belongs to $\mathcal T$ whenever $\mathsf u \in \mathcal T$.

We claim that $\mathsf {F}$ is contractive. It sufficient to check that for any $\mathsf u, \mathsf v\in \mathcal T$ and any $n,k\in\mathbb N$ we have 
\[\mathsf u\sim_{n,k}\mathsf v \Rightarrow \mathsf {F(u)}\sim_{n,k+1}\mathsf{F(v)}.\] which is done by case analysis.

Now we define the required cut-elimination mapping $\mathsf{ce}$ as the fixed-point of the the operator $\mathsf F \colon \mathcal T \to \mathcal T$, which exists by Lemma \ref{Comp T} and Theorem \ref{fixpoint}.

For some $n,k\in\mathbb N$, let us call a mapping $\mathsf u\in \mathcal T$ $(n,k)$-free if it satisfies the following two conditions: $\mathsf u(\pi)\in\mathcal P_n$ for any $\pi \in \mathcal P$; $\mathsf u(\pi)\in\mathcal P_{n+1}$ whenever $\lvert \pi\rvert<k$. 

%We define $$\mathcal N_{n,k}:= \{\mathsf{u}\in \mathcal T\: \mid \: \text{ for all $\pi\in\mathcal P$, }\mathsf{u}(\pi) \in \mathcal{P}_n\text{ and if $\lvert\pi\rvert< k$, }\mathsf u(\pi)\in \mathcal P_{n+1}\}.$$ 

%By the straightforward case analysis according to the definition of $\mathsf G_{\Box B}$, we see that for any $(n,k)$-adequate $\mathsf u\in\mathcal R_{\Box B}$ the mapping $\mathsf G_{\Box B}(\mathsf u)$ is $(n,k+1)$-adequate. 

By case analysis we established that $\mathsf F(\mathsf u)$ is $(n,k+1)$-free for any $(n,k)$-free $\mathsf u\in\mathcal T$. Notice that if a mapping $\mathsf{u} $ is $(n,k)$-free for all $k\in\mathbb N$, then it is also $(n+1,0)$-free. Now by induction on $n$ with a subinduction on $k$, we immediately obtain that the mapping $\mathsf {ce}$, which is a fixed-point of $\mathsf F$, is $(n,k)$-free for all $n,k\in\mathbb N$. Therefore, for any $\infty$-proof $\pi$, the $\infty$-proof $\mathsf {ce}(\pi)$ does not contain instances of the rule $\mathsf {(cut)}$.

Now assume $\mathsf{Grz_\infty} +\mathsf{cut}\vdash\Gamma\Rightarrow\Delta$. Take an $\infty$-proof of the sequent $\Gamma\Rightarrow\Delta$ in the system $\mathsf{Grz_\infty} +\mathsf{cut}$ and apply the mapping $\mathsf{ce}$ to it. We obtain an $\infty$-proof of the same sequent in the system $\mathsf{Grz_\infty}$.
\end{proof}
Theorem \ref{cutelimgrz} is now established as a direct consequence of Theorem \ref{seqtoinfcut}, Theorem \ref{infcuttoinf}, and Theorem \ref{inftoseq}.

\section{Conclusions and Future Work}
We have proven the cut elimination theorem for the logic $\mathsf{Go}$ syntacticly by constructing a continuous cut eliminating mapping for proofs in a system allowing non-well-founded proofs. %Because of similarities of the given proof to the case of the logic $\mathsf{Grz}$, we think that this approach is general and flexible enough to provide a nice framework for proving proof theoretic properties of various logics in a uniform way.
This method seems to provide uniform approach to cut-elimination in different logics. Indeed --- we can write systems for logics $K4$, $GL$, and $Grz$ by slightly modifying the system $\mathsf{Go_{Seq}}$, define appropriate ultrametrics on them, and leave the proof of cut elimilation virtually unchanged. The next step is to take this method to more complicated logics like the logic of transitive closure. 
%\section{Acknowledgements}
%The author thanks D. Shamkanov for useful input and support.

\bibliographystyle{aiml18}
\bibliography{aiml18}

%% Appendix.
%% Remove the \Appendix command if an 
%% appendix is not required.
\Appendix
Here we recall basic notions of the theory of ultrametric spaces (cf. \cite{Shor}) and consider several examples concerning $\infty$-proofs.% and define an ultrametric on the set of all $\infty$-proofs.

An \emph{ultrametric space} $(M,d)$ is a metric space that satisfies a stronger version of the triangle inequality: for any $x,y,z\in M$
$$d(x,z) \leqslant \max \{d(x,y), d(y,z)\}.$$ 
%For $x,y \in M$ and $r\in [0, + \infty)$, we write $x \equiv_r y$ if $d(x,y)\leqslant r$. 

For $x\in M$ and $r\in [0, + \infty)$, the set $B_r(x)=\{ y\in M \mid d(x,y) \leqslant r\}$ is called the \emph{closed ball} with \emph{center $x$} and \emph{radius $r$}. Recall that a metric space $(M,d)$ is \emph{complete} if any descending sequence of closed balls, with radii tending to $0$, has a common point. An ultametric space $(M,d)$ is called \emph{spherically complete} if an arbitrary descending sequence of closed balls has a common point.

For example, consider the set $\mathcal P$ of all $\infty$-proofs of the system $\mathsf{Go}_{\infty} +\mathsf{cut} $. We can define an ultrametric $d_{\mathcal P} \colon {\mathcal P} \times {\mathcal P} \to [0,1]$ on $\mathcal P$ by putting
\[d_{\mathcal P}(\pi,\tau) = 
\inf\{\frac{1}{2^n} \mid  \pi \sim_n \tau\}.
\]
%\begin{cases} 
%0 & \text{if $\pi=\tau$,}\\
%2^{- \sup\{n \in \mathbb{N} \:\mid  \pi \sim_n \tau\}} & \text{otherwise.}
%\end{cases}
We see that $d_{\mathcal P}(\pi , \tau) \leqslant 2^{-n}$ if and only if $\pi \sim_n \tau$. Thus, the ultrametric $d_{\mathcal P}$ can be considered as a measure of similarity between $\infty$-proofs.

\begin{propos}[\ref{ComplP}]\label{ComplP1}
$(\mathcal{P},d_{\mathcal P})$ is a (spherically) complete ultrametric space.
\end{propos}

Consider the following characterization of spherically complete ultrametric spaces. Let us write $x \equiv_r y$ if $d(x,y)\leqslant r$.
%All our ultrametric spaces will be \emph{$1$-bounded} that is $d(x,y) \leqslant 1$ for any $x,y\in M$.
Trivially, the relation $\equiv_r$ is an equivalence relation for any ultrametric space and any number $r\geqslant 0$. %In addition, $\equiv_{r_1} \subset\equiv_{r_2}$ if $r_1 <r_2$. Also, $\bigcap \{\equiv_r \mid r \in [0, + \infty) = =$.
%Now we have the following characterization of spherically complete ultrametric spaces.

\begin{propos}\label{sphcomp}
An ultametric space $(M,d)$ is spherically complete if and only if for any sequence $(x_i)_{i\in \mathbb N}$ of elements of $M$, where $x_i\equiv_{r_i}x_{i+1}$ and  
%of the form $x_0 \equiv_{r_0} x_1 \equiv_{r_1} \dots$ with 
$r_i \geqslant r_{i+1}$ for all $i\in \mathbb N$, there is a point $x$ of $M$ such that $x \equiv_{r_i} x_i$ for any $i\in\mathbb N$.
\end{propos}
\begin{proof}
($\Rightarrow$) Assume $(M,d)$ is a spherically complete ultrametric space. Consider a sequence $(x_i)_{i\in \mathbb N}$ of elements of $M$ such that $x_i\equiv_{r_i}x_{i+1}$ and $r_i \geqslant r_{i+1}$ for all $i\in \mathbb N$.  Then the sequence $(B_{r_i}(x_i))$ is a descending sequence of closed balls, and therefore by spherical completeness has a common point $x$. Trivially, the point $x$ satisfies the desired conditions.

($\Leftarrow$) Assume there is a descending sequence of closed balls $(B_{r_i}(x_i))$. We have that $x_0 \equiv_{r_0} x_1 \equiv_{r_1} \dots$ and $r_i \geqslant r_{i+1}$ for all $i\in \mathbb N$. So there is an element $x\in M$ such that $x \equiv_{r_i} x_i$, meaning it lies in all the balls.
\end{proof}

In an ultrametric space $(M,d)$, a function $f \colon M\to M$ is called \emph{non-expansive} if $d(f(x),f(y)) \leqslant d(x,y)$ for all $x,y\in M$.
For ultrametric spaces $(M, d_M)$ and $(N, d_N)$, the Cartesian product $M \times N$ can be also considered as an ultrametric space with the metric $$d_{M \times N} ((x_1,y_1),(x_2,y_2)) = \max \{d_M(x_1,x_2), d_N(y_1,y_2) \}.$$ 

%In an ultrametric space $(M,d)$, let us call a function $f \colon M^n\to M$  \emph{non-expansive} if $d(f(x_1,\ldots,x_n),f(y_1,\ldots,y_n)) \leqslant \max_{i\leqslant n} d(x_i,y_i)$ for all $x_i,y_i\in M$.

Let us consider another example. For $m\in \mathbb N$, let $\mathcal F_m$ denote the set of all non-expansive functions from $\mathcal P^m$ to $\mathcal P$. Note that any function $\mathsf u\colon \mathcal P^m\to\mathcal P$ is non-expansive if and only if for any tuples $\vec\pi$ and $\vec\pi'$, and any $n\in\mathbb N$ we have 
$$\pi_1\sim_n\pi'_1,\dotsc,\pi_m\sim_n\pi'_m\Rightarrow \mathsf u(\vec\pi)\sim_n\mathsf u(\vec\pi').$$
%For two mappings $\mathsf a,\mathsf b\mathcal \colon \mathcal P\to\mathcal P$ we write that $\mathsf a\sim_{n,k}\mathsf b$, where $n,k\in\mathbb N$, if for every $\pi\in \mathcal P$ we have $\mathsf a(\pi)\sim_n\mathsf b(\pi)$ and for every $\pi\in P$ with $\lvert\pi\rvert< k$ we have $\mathsf a(\pi) \sim_{n+1}\mathsf b(\pi).$
Now we introduce an ultrametric for $\mathcal F_m$. For $\mathsf{a,b}\in \mathcal F_m$, we write 
$\mathsf a\sim_{n,k}\mathsf b$ if $\mathsf a(\vec{\pi})\sim_n\mathsf b(\vec{\pi})$ for any $\vec{\pi}\in\mathcal P^m$ and, in addition, $\mathsf a(\vec{\pi}) \sim_{n+1}\mathsf b(\vec{\pi})$ whenever $\Sigma_{i=1}^m\lvert\pi_i\rvert< k$.\footnote{This definition is inspired by \cite[Subsection 2.1]{GiMi}.}
%$\mathsf a\sim_{n}\mathsf b$ if for any $\vec{\pi}\in\mathcal P^m$ we have $\mathsf a(\vec{\pi})\sim_n\mathsf b(\vec{\pi})$
%$\mathsf a\sim_{n,k}\mathsf b$ if $\mathsf a(\vec{\pi})\sim_n\mathsf b(\vec{\pi})$ for every $\pi_1,\ldots,\pi_m\in\mathcal P$, and $\mathsf a(\pi_1,\ldots,\pi_m) \sim_{n+1}\mathsf b(\pi_1,\ldots,\pi_m)$ whenever $\Sigma_{i=1}^m\lvert\pi_i\rvert< k$.
An ultrametric $l_m$ on $\mathcal F_m$ is defined by 
\[l_m(\mathsf a,\mathsf b)=\frac{1}{2}\inf\{\frac{1}{2^n}+\frac{1}{2^{n+k}}\mid \mathsf a\sim_{n,k}\mathsf b\}.\]
We see that $l_m(\mathsf a,\mathsf b) \leqslant 2^{-n-1}+2^{-n-k-1}$ if and only if $\mathsf a\sim_{n,k}\mathsf b$.

\begin{propos}[\ref{SphCom}]
$(\mathcal F_m, l_m)$ is a spherically complete ultrametric space.
\end{propos}
\begin{proof}
Assume we have a series
$\mathsf a_0\sim_{n_0,k_0}\mathsf a_1\sim_{n_1,k_1}\dotso  $,
where the sequence $r_i=2^{-n_i}+2^{-n_i-k_i-1}$ is non-increasing. From Proposition \ref{sphcomp}, it is sufficient to find a function $\mathsf a\in\mathcal F_m$ such that $\mathsf a\sim_{n_i,k_i}\mathsf a_i$ for all $i\in \mathbb N$.

%If that sequence stabilizes, then there is $j\in\mathbb N$ such that $\forall i>j (n_i,k_i)=(n_j,k_j)$, and we can take $\mathsf a_j$ as $\mathsf a$.

Suppose $\lim_{i\to\infty}r_i=0$. Consider a tuple $\vec{\pi}\in\mathcal P^m $. We have that $\lim_{i\to\infty}n_i=+ \infty$ and $\mathsf a_0 (\vec{\pi})\sim_{n_0}\mathsf a_1 (\vec{\pi})\sim_{n_1}\dotso$ .
By Proposition \ref{ComplP1}, there is an $\infty$-proof $\tau$ such that $\tau \sim_{n_i} \mathsf a_i (\vec{\pi})$ for all $i\in \mathbb N$. We define $\mathsf a (\vec{\pi}) = \tau$. We need to check that the mapping $\mathsf a$ is non-expansive. If for tuples $\vec\pi$ and $\vec\pi'$ we have $\pi_1\sim_n\pi'_1,\ldots,\pi_m\sim_n\pi'_m$ for some $n\in\mathbb N$, then we can choose $i$ such that $n_i>n$. We have $\mathsf a(\vec\pi)\sim_n \mathsf a_i(\vec\pi)\sim_n\mathsf a_i(\vec\pi')\sim_n\mathsf a(\vec\pi')$. Therefore $\mathsf a(\vec\pi)\sim_n \mathsf a(\vec\pi')$ and the mapping $\mathsf a$ is non-expansive.

If $\lim_{i\to\infty}r_i>0$, then $\lim_{i\to\infty}n_i= n$ for some $n \in \mathbb{N}$. We have two cases: either $\lim_{i\to\infty}k_i= k$ for a number $k \in \mathbb{N}$, or $\lim_{i\to\infty}k_i= +\infty$. In the first case, there is $j\in\mathbb N$ such that $(n_i,k_i)=(n_j,k_j)$ for all $i>j$, and we can take $\mathsf a_j$ as $\mathsf a$. Here the mapping $\mathsf a$ is obviously non-expansive.
In the second case, for a tuple $\vec{\pi}$ we define $\mathsf a (\vec{\pi})$ to be $\mathsf{a}_j (\vec{\pi})$, where $j= \min \{ i \in \mathbb{N} \mid n_i= n \text{ and } \sum_{s=1}^m\lvert\pi_s\rvert < k_i \}$. For all tuples $\vec\pi$ and $\vec\pi'$ we have $\mathsf a(\vec\pi)\sim_0 \mathsf a(\vec\pi')$. If for tuples $\vec\pi$ and $\vec\pi'$ and some $n\geqslant 1$ we have $\pi_1\sim_n\pi'_1,\ldots,\pi_m\sim_n\pi'_m$, then $\sum_{s=1}^m\lvert\pi_s\rvert=\sum_{s=1}^m\lvert\pi_s'\rvert=t$ and $\mathsf a(\vec\pi)=\mathsf a_j(\vec\pi)\sim_n \mathsf a_j(\vec\pi')=\mathsf a(\vec\pi')$, where $j=\min \{ i \in \mathbb{N} \mid n_i= n \text{ and } t < k_i \}$. Therefore $\mathsf a(\vec\pi)\sim_n \mathsf a(\vec\pi')$ and the mapping $\mathsf a$ is non-expansive.

%If that sequence stabilizes, then there is $j\in\mathbb N$ such that $\forall i>j (n_i,k_i)=(n_j,k_j)$, and we can take $\mathsf a_j$ as $\mathsf a$.

%For any appropriate tuple of $\infty$-proofs $(\pi_1,\ldots,\pi_m)$ if for some $j_0>i_0$ we have $k_{j_0}>\sum_{i=1}^m\lvert\pi_i\rvert$, then for every $j_1,j_2>j_0$ we have $\mathsf a_{j_1}(\pi_1,\ldots,\pi_m)\sim_{n_{i_0}+1}\mathsf a_{j_2}(\pi_1,\ldots,\pi_m)$.

%This allows us to define the mapping $\mathsf a$ the following way:
%$$\mathsf a(\pi_1,\ldots,\pi_m)=\mathsf a_i(\pi_1,\ldots,\pi_m),\text{ where } i=\min\{j\mid j>i_0\text{ and }k_j>\sum_{i=1}^m\lvert\pi_i\rvert\}.  $$

%Clearly  $\mathsf a\sim_{n_i,k_i}\mathsf a_i$ for all $i>i_0$, and therefore $(\mathcal F, l_{\mathcal P})$ is a spherically complete ultrametric space.

\end{proof}

In an ultrametric space $(M,d)$, a function $f\colon M \to M$ is called \emph{(strictly) contractive} if $d(f(x),f(y)) < d(x,y)$ when $x \neq y$.

%Notice that any operator $\mathsf U\colon\mathcal F_m\to \mathcal F_m$ is strictly contractive if and only if for any $\mathsf a,\mathsf b\in\mathcal F_m$, and any $n,k\in\mathbb N$ we have 
%$$\mathsf a\sim_{n,k}\mathsf b\Rightarrow \mathsf U(\mathsf a)\sim_{n,k+1}\mathsf U(\mathsf b).$$

%Now we state a generalization of the Banach's fixed-point theorem for ultrametric spaces that will be used in the next sections.

%\begin{theorem}[Prie\ss-Crampe \cite{PrCr}]
%Let $(M,d)$ be a non-empty spherically complete ultrametric space. Then every strictly contractive mapping $f\colon M\to M$ has a unique fixed-point. 
%\end{theorem}

%% Bibliography
%% Make sure to use the bibliographystyle aiml18.

%\bibliographystyle{aiml18}
%\bibliography{aiml18}

\end{document}